\documentclass{amsart}
\usepackage{amsmath, amsthm, amscd, amsfonts, amssymb,latexsym}
 \newtheorem{thm}{Theorem}[section]

 \newtheorem{prop}[thm]{Proposition}
 \theoremstyle{definition}
 
 \newtheorem{oq}[thm]{Open Question}
 
 \theoremstyle{remark}
 \newtheorem{rem}[thm]{Remark}
 \newtheorem{ex}[thm]{Example}
 \numberwithin{equation}{section}

\begin{document}

\title{Permutability degrees of finite groups}

\author[D.E. Otera]{D.E. Otera$^\dagger$}

\author[F.G. Russo]{F.G. Russo$^{\dagger\dagger}$}

\date{\today}

\thanks{$^\dagger$ Institute of Mathematics and Informatics, Vilnius University, Vilnius, Lithuania;
email: daniele.otera@gmail.com\\
$^{\dagger\dagger}$ Department of Mathematics and Applied Mathematics, University of Cape Town, Cape Town, South Africa; email: francescog.russo@yahoo.com}

\begin{abstract}Given a finite group $G$, we introduce the \textit{permutability
degree} of $G$, as $$pd(G)=\frac{1}{|G| \ |\mathcal{L}(G)|}
{\underset{X \in \mathcal{L}(G)}\sum}|P_G(X)|,$$ where
$\mathcal{L}(G)$ is the subgroup lattice of $G$ and $P_G(X)$ the
permutizer of the subgroup $X$ in $G$, that is, the subgroup
generated by all cyclic subgroups of $G$ that permute with $X\in
\mathcal{L}(G)$. The number $pd(G)$ allows us to find some
structural restrictions on $G$. Successively, we investigate the
relations between $pd(G)$, the probability of commuting subgroups
$sd(G)$ of $G$ and the probability of commuting elements $d(G)$ of
$G$. Proving some inequalities between $pd(G)$, $sd(G)$ and $d(G)$, we
correlate these  notions. 
\end{abstract}

\maketitle


\section{Introduction}
All the groups of the present paper are supposed to be finite.
Given a group $G$ and its subgroup lattice $\mathcal{L}(G)$, the
\textit{subgroup commutativity degree}
$$sd(G)=\frac{|\{(H,K)\in \mathcal{L}(G) \times \mathcal{L}(G) \ | \ HK =
KH\}|}{|\mathcal{L}(G)|^2} $$ of $G$ and the \textit{commutativity degree}
$$d(G)=\frac{|\{(x,y)\in G \times G \ | \ xy = yx\}|}{|G|^2}$$ of $G$ have been largely
studied  in the last years. Fundamental properties and interesting
generalizations of $sd(G)$ can be found in \cite{far, far2, far3,
otera, mt, mtbis, mttris}, and for $d(G)$ in \cite{alghamdibis, alghamdi,
britnell, elr, er3,  er4, gr, hofmann-russo1, hofmann-russo2, les1,
les2, russo1}. To study these notions,
various perspectives have been considered in literature, because
both measure theory and combinatorial techniques may be applied in
order to get restrictions on the structure of a group.

The present paper investigates a similar concept, the
$permutability$ $degree$ of $G$
$$ pd(G)= \frac{1}{|G| \ |\mathcal{L}(G)|}{\underset{X\in
\mathcal{L}(G)}\sum}|P_G(X)|$$
and its connections with  $sd(G)$ and $d(G)$. In the previous
formula, the $permutizer$ $P_G(X)$ of a subgroup $X$  of $G$ is defined to
be the subgroup generated by all cyclic subgroups of $G$ that
permute with $X$, that is,
$P_G(X)= \langle g \in G \ | \ \langle g \rangle X = X \langle g \rangle \rangle.$
This means that $X \in \mathcal{L}(P_G(X))$ and $X\not=P_G(X)$ if and only if $X\langle g \rangle =\langle g \rangle X$ for some $g \in
G - X$.

We concentrate on permutizers because several classifications are
available in literature on this topic. Recall that a group $G$
such that $X\not= P_G(X)$ for every proper subgroup $X$ of $G$ is
said to satisfy the $permutizer$ $condition$ $\mathbf{P}$, or
briefly  $\mathbf{P}$--$group$.  Therefore the permutizer
condition generalizes the well--known $normalizer$ $condition$
(see \cite{rs}) and gives information on how the group is near to
be supersolvable. The study of  permutizers is not new and
it is based on a series of fundamental contributions
\cite{br,  wang1, wang2, zhang} in the last 20 years.
From \cite[Corollary 2]{br}, we know that for groups of odd order
the permutizer condition is equivalent of being supersolvable and
actually, a complete classification of $\mathbf{P}$--groups can be
found in \cite{br}.

Now we may define the subgroup
$$P(G)={\underset{H \in \mathcal{L}(G)}\bigcap P_G(H)}$$
and correlate it with other subgroups of $G$. For instance, it is
easy to check that the \textit{norm} $N(G)$ of $G$  (see \cite{rs}
for the properties of $N(G)$) satisfies the following relation
$$Z(G)= {\underset{x \in G}\bigcap C_G(x)} \subseteq
N(G)={\underset{H \in \mathcal{L}(G)}\bigcap N_G(H)}\subseteq
{\underset{H \in \mathcal{L}(G)}\bigcap P_G(H)}=P(G).$$ This
relation emphasizes how $P(G)$ is connected with other subgroups,
widely investigated in literature, such as intersections of
normalizers or of centralizers. Note that for any $X \in \mathcal{L}(P(G))$ one has $P_G
(X)=G$.

The subgroup $P(G)$ is also important because it allows us to
``manipulate" the expression of $pd(G)$, for getting some
analogies with
$$sd(G)=\frac{1}{|\mathcal{L}(G)|^2} \
{\underset{H\in \mathcal{L}(G)}\sum}|\mathcal{C}_{\mathcal{L}(G)}(H)| \ \ \mathrm{and} \ \
d(G)=\frac{1}{|G|^2} \sum_{x \in G}|C_G(x)|,$$ where $P_G(X)$ is the natural substitute of $\mathcal{C}_{\mathcal{L}(G)}(X)=\{Y \in \mathcal{L}(G) \ | \
YX=XY\}$ in \cite{otera,mt} and of  $C_G(x)=\{y \in G \ | \ xy=yx\}$  in \cite{alghamdi,elr}.  This manipulation of the expression of $pd(G)$ will allow us to detect whether $P(G)$ is cyclic or not by looking only at the size of $pd(G)$.

\section{Basic properties and terminology}
Some of the following observations will be useful later on.

\begin{rem}\label{rem0}Since  $P_G(X)$ is a subgroup of $G$,
for all $X \in \mathcal{L}(G)$, and it always contains the trivial
subgroup, then $|P_G(X)|\le |G|$ and also $$0< {\underset{X\in
\mathcal{L}(G)}\sum}|P_G(X)|\le |\mathcal{L}(G)| \ |G|$$ so that
$pd(G)\in ]0,1]$.\end{rem}

\begin{rem}\label{rem1}A group $G$ has $pd(G)=1$ if and only if the sum
of all $|P_G(X)|$ for $X \in \mathcal{L}(G)$ is equal to
$|G||\mathcal{L}(G)|$. By default, a \textit{quasihamiltonian
group} $G$, that is, a group in which every subgroup is
permutable, has $pd(G)=1$. A classification of quasihamiltonian
groups can be found in \cite[Theorems 2.4.11 and 2.4.16]{rs} and,
roughly speaking, these groups are direct products of abelian
groups by a copy of the quaternion group of order 8. In particular, abelian groups have permutability degree equal to
1.
\end{rem}

Another case in which the permutability degree reaches 1 is the
following.

\begin{rem}\label{rem2}A $\mathbf{P}$--group $G$ in
which all proper subgroups are maximal has $pd(G)=1$. In such a
case  for all proper subgroups $X$ of $G$ one has $X \subset
P_G(X)=G$ and so $pd(G)=1$. One might be tempted to think that
 all $\mathbf{P}$--groups have permutability degree equal to 1, but
Example \ref{d8} below shows this is false, and then the
additional condition ``in which all proper subgroups are maximal"
cannot be omitted.
\end{rem}

Now we  rewrite the original expression
of permutability degree in the following more useful form. Since
$\mathcal{L}(P(G))$ is a sublattice of $\mathcal{L}(G)$, it turns
out that
\begin{equation}\label{e:1} pd(G)= \frac{1}{|G||\mathcal{L}(G)|} \ \left({\underset{X\in
\mathcal{L}(P(G))}\sum}|P_G(X)| + {\underset{X\in \mathcal{L}(G)-
\mathcal{L}(P(G))}\sum}|P_G(X)| \right).
\end{equation}

Note also that a cyclic group $C$ (or better a quasihamiltonian
group $Q$)  has $sd(C)=pd(C)=d(C)=1$ (or better $sd(Q)=pd(Q)=1$).
Therefore, relations between $sd(G)$ and $pd(G)$ are meaningful
when $G$ is noncylic and nonquasihamiltonian (see Remark
\ref{rem1}).

For the sake of completeness, we recall some results of Beidleman and
Heineken in \cite{bh}. The $quasicenter$ $Q(G)$ of $G$ is the
subgroup of $G$ generated by all elements $g\in G$ such that
$\langle g\rangle K=K\langle g \rangle$, where $K$ is an arbitrary
subgroup of $G$. The subgroup $Q(G)$ was introduced by Mukherjee
and studied by several authors in the last years (see \cite{bh,
r1,  r3}), who investigated chains of quasicenters and relations
with supersolvable groups. On the other hand, the
$hyperquasicenter$ of $G$, denoted by $Q_\infty(G)$, is the
largest term of the chain $1=Q_0(G)\leq Q_1(G)=Q(G)\leq \ldots\leq
Q_i(G)\leq Q_{i+1}(G)\leq\ldots$ of normal subgroups of $G$,
where, for any $i\geq0$, $Q_{i+1}(G)/Q_i(G)=Q(G/Q_i(G))$ and
$Q_\infty(G)={\underset{i\geq0} \bigcup}Q_i(G).$

Recall that a normal subgroup $N$ of $G$ is said to be
$hypercyclically$ $embedded$ in $G$ if it contains a
$G$--invariant series whose factors are cyclic. It is easy to see
that $G$ contains a unique largest hypercyclically embedded
subgroup, which we  denote $\Sigma(G)$. More precisely,
\cite[Theorem 1]{bh} shows that $\Sigma(G)=Q_\infty(G)$ is true
for any group $G$. Some interesting connections hold between
$\mathbf{P}$--groups, $P(G)$ and $Q_\infty(G)$. For instance,
\cite[(3.1), p. 697]{br} shows that a group $G$ is a
$\mathbf{P}$--group if and only if $G/\Sigma(G)$ is a
$\mathbf{P}$--group. As a first consequence, a group $G$ is a
$\mathbf{P}$--group if and only if $G/Q_\infty(G)$ is a
$\mathbf{P}$--group. As a second consequence, $Z(G)\subseteq
Q(G)\subseteq P(G)$ is true for any group $G$. Furthermore,
$Q_\infty(G)=P(G)$ if and only if $P(G)=\Sigma(G)$.

\section{Examples}

Now we specify some of the previous notions for the symmetric group $S_3$
on 3 objects. This will help us  to visualize analogies and
differences between permutability degrees, subgroup commutativity
degrees and commutativity degrees.

\begin{ex}\label{d6}
The  smallest nonabelian group $S_3$  has
$\mathcal{L}(S_3)=\{\{1\}, S_3, A_3, H,K,L\}$, where 
$A_3=\langle (123)\rangle=\langle a \rangle$,  $H=\langle(12)\rangle=\langle h \rangle$,  $K=\langle (13)\rangle=\langle k \rangle$,  $L=\langle
(23)\rangle=\langle l \rangle$. Noting that $HK\not=KH$,
$HL\not=LH$, $KL\not=LK$, one has
$$P_{S_3}(\{1\})=P_{S_3}(A_3)=P_{S_3}(S_3)=S_3; $$
$$  \ \ P_{S_3}(H)=P_{S_3}(K)=P_{S_3}(L)=S_3 \ ; \\
P(S_3)=S_3.$$ The fact that $P_{S_3}(\{1\})=P_{S_3}(A_3)=P_{S_3}(S_3)=S_3$ is clear, since we are dealing with permutizers of
normal subgroups. On the other hand, $A_3=\langle a \rangle
\subseteq P_{S_3}(H)$ by definition. Now, $H\subseteq P_{S_3}(H)$
is obvious. Then $HA_3$, which is a subgroup of $S_3$, should be
contained in $P_{S_3}(H)$ and $|HA_3|=\frac{|H| \ |A_3|}{|H \cap
A_3|}=6$. This forces $ P_{S_3}(H)$ to be equal to $S_3.$
Similarly we get $P_{S_3}(K)=P_{S_3}(L)=S_3$.

\setlength{\unitlength}{10mm}
\begin{picture}(4,4)(-1,-1.5)
\put(1.6,2.1){$S_3$}
\put(2,2){\line(-1,-1){}}
\put(.5,.8){$A_3$}
\put(2,2){\line(0,-1){}}
\put(1,.8){\line(1,-1){}}
\put(1.6,-.5){$\{1\}$}
\put(2,.8){$H$}
\put(2,2){\line(1,-1){}}
\put(2,-.2){\line(0,1){}}
\put(2,-.2){\line(1,1){}}
\put(2,-.2){\line(2,1){2}}
\put(3,.8){$K$}
\put(2,2){\line(2,-1){2}}
\put(4,.8){$L$}
\put(5,.8){\textbf{Fig.III.1.} Hasse diagram of $\mathcal{L}(S_3)$.}
\end{picture}

It is interesting to note that an
example as easy as this has a lot of
properties in our perspective of study. The group $S_3$ is
supersovable by looking at the series $\{1\}\triangleleft A_3
\triangleleft S_3$, but is not quasihamiltonian, due to
$HK\not=KH$. At the same time, $S_3$ does not satisfy the
normalizer condition, since it is not nilpotent. Moreover, $S_3$
has $Z(S_3)=\{1\}$, $\Sigma(S_3)=Q_2(S_3)=Q_\infty(S_3)=S_3$,
$Q(S_3)=A_3$ and it is a $\mathbf{P}$--group, since
$S_3/\Sigma(S_3)=\{1\}$ is obviously a $\mathbf{P}$--group.

A direct calculation shows that
$$6 \cdot 6 \ pd(S_3)={\underset{X \in \mathcal{L}(S_3)}\sum}|P_{S_3}(X)| = |P_{S_3}(H)|+|P_{S_3}(K)|$$
$$+|P_{S_3}(L)|+|P_{S_3}(A_3)|+|P_{S_3}(S_3)|+|P_{S_3}(\{1\})|= 36,$$
then  $pd(S_3)=1 >  sd(S_3)=\frac{5}{6} >  \frac{1}{2}=d(S_3)$, agreeing with the computations in \cite[p.2510]{mt} and \cite{elr, er3}.
\end{ex}

Another easy (but interesting) example is the following.

\begin{ex}\label{d8}
The dihedral group of order 8 is
$D_8=\langle a,b \ | \ a^2=b^4=1, a^{-1}ba=b^{-1}\rangle $
and  has
$$\mathcal{L}(D_8)=\{\{1\}, \langle b \rangle, \langle b^2\rangle, \langle a \rangle, \langle ba \rangle, \langle b^2a \rangle, \langle b^3a \rangle, \{1,b^2,a,b^2a\}, \{1,b^2,ba,b^3a\}, D_8 \}.$$
The normal subgroups are $D_8$, $\{1\}$, $B=\langle b \rangle$, $Z(D_8)=\langle b^2 \rangle$, $M_1=\{1,b^2,a,b^2a\}$ and $M_2=\{1,b^2,ba,b^3a\}$. Notice that $H=\langle b^2a \rangle$ and $K=\langle a\rangle$ are contained in $M_1$, while $U=\langle ba\rangle $ and $V= \langle b^3a \rangle$ in $M_2$.

\setlength{\unitlength}{10mm}
\begin{picture}(10,4)(-.8,-1.5)
\put(1.6,2.1){$D_8$}
\put(2,2){\line(-1,-1){}}
\put(.6,.8){$M_1$}
\put(2,2){\line(0,-1){}}
\put(2,-1.5){\line(-1,1){1.3}}
\put(2,-1.5){\line(1,1){1.1}}
\put(.6,.8){\line(0,-1){.7}}
\put(3.1,.8){\line(0,-1){.7}}
\put(1,.8){\line(1,-1){}}
\put(1.6,-.5){$Z(D_8)$}
\put(2,.8){$B$}
\put(2,-1.6){\line(-2,1){2.5}}
\put(-.4,-.1){\line(1,1){1}}
\put(-.6,-.3){$H$}
\put(0.5,-.2){$K$}
\put(3,-.3){$U$}
\put(4.2,-.4){$V$}
\put(2,-1.6){\line(2,1){2.3}}
\put(4.2,-.2){\line(-1,1){1}}
\put(2,2){\line(1,-1){}}
\put(2,-.2){\line(0,1){}}
\put(2,-.2){\line(1,1){}}
\put(3,.8){$M_2$}
\put(2,-.5){\line(0,-1){}}
\put(1.8,-1.8){$\{1\}$}
\put(5,.8){\textbf{Fig.III.2.} Hasse diagram of $\mathcal{L}(D_8)$.}
\end{picture}

Moreover
$$8=|D_8|=|P_{D_8}(\{1\})|=|P_{D_8}(D_8)|=|P_{D_8}(\langle b \rangle)|=|P_{D_8}(Z(D_8))|$$
$$=|P_{D_8}(M_1)|=|P_{D_8}(M_2)|,$$     $$4=|M_1|=|P_{D_8}(H)|=|P_{D_8}(K)|,  \  \ 4=|M_2|=|P_{D_8}(V)|=|P_{D_8}(U)|$$
and $Q(D_8)=P(D_8)=Z(D_8)<\Sigma(D_8)=Q_\infty(D_8)=D_8.$ In fact, $D_8=\Sigma(D_8)$ is supersolvable and it is
also a $\mathbf{P}$--group, but nevertheless its permutability
degree is different from 1, because
$$8 \cdot 10 \ pd(D_8)={\underset{X \in \mathcal{L}(D_8)}\sum}|P_{D_8}(X)|
= 6 \cdot |D_8|+ 2 \cdot |\{1,b^2,a,b^2a\}| + 2 \cdot |\{1,b^2,ba,b^3a\}|=64.$$
More precisely,
$$d(D_8)=\frac{5}{8} < pd(D_8)=\frac{64}{80}=\frac{4}{5} < sd(D_8)=\frac{46}{55}.$$
The value of $d(D_8)$ can be found in \cite{elr} and that of
$sd(D_8)$ in \cite{mt}. This example shows that there exist
$\mathbf{P}$--groups with permutability degree different from 1.
Note  that $D_8$ satisfies $8=|D_8|<|\mathcal{L}(D_8)|=10$ but
$$8=|\{\{1\}, D_8, \langle b \rangle, \{1,b^2,ba,b^3a\},
\{1,b^2,a,b^2a\}, \langle b^2a\rangle, \langle b^2 \rangle,
\langle a \rangle\}|=|\mathcal{C}_{\mathcal{L}(D_8)}(\langle a
\rangle)|$$ $$\not \le  |\mathcal{Z}_{\mathcal{L}(D_8)}(\langle a \rangle)|=|\{\{1\},\langle b^2a\rangle, \langle b^2 \rangle, \langle a \rangle \}|=4.$$
\end{ex}

Examples \ref{d6} and \ref{d8} illustrate a series of problems for
the computation of the permutability degree, arising from the
nature of the subgroup lattice of the groups under consideration. 
We will come back to this point later on.

\section{General properties of the permutability degree}

We note that \cite[Theorems 2.5, 3.3]{elr} shows that the
commutativity degree is monotone. This is a well--known property,
which is due to the fact that we are dealing with a positive
monotone measure of probability. Similar situations can be found
for $sd(G)$ in  \cite[Proposition 2.4, Corollaries 2.5, 2.6, 2.7,
Theorems 3.1.1, 3.1.5]{mt} and in \cite{er3, gr, otera}.
For $pd(G)$ we have something similar.

\begin{thm}\label{t:monotone}  Let  $H$ be a subgroup of a group $G$.
Then
$$\frac{|\mathcal{L}(H)|}{|\mathcal{L}(G)| \ |G:H|}  \ pd(H) \le  \ pd(G).$$
Moreover, if $|P_G(X): P_H(X)|\leq |G:H|$ for all $X \in \mathcal{L}(G)$, $P(G) \le H$ and $|\mathcal{L}(G)-\mathcal{L}(P(G))| \le |\mathcal{L}(P(G))|$, then $|\mathcal{L}(G)| \ pd(G) \le 2 |\mathcal{L}(H)| \ pd(H)$. In particular,
$$  \frac{|\mathcal{L}(H)|}{|\mathcal{L}(G)| \ |G:H|}  \ pd(H) \le   \ pd(G)  \le \frac{2|\mathcal{L}(H)|}{|\mathcal{L}(G)|} pd(H).$$
\end{thm}

\begin{proof} We start by proving the first inequality.
Since $|P_H(X)|\le |P_G(X)|$ for all $X \in \mathcal{L}(G)$, we
have $$|G| \ |\mathcal{L}(G)|  \ pd(G) = \sum_{X \in
\mathcal{L}(G)}|P_G(X)|$$
$$\ge \sum_{X \in
\mathcal{L}(G)}|P_H(X)|=\sum_{X \in \mathcal{L}(H)}|P_H(X)|=
pd(H) \ |H| \ |\mathcal{L}(H)|$$ and the result follows.

Now we prove the second inequality. Since by hypothesis,
$|P_G(X): P_H(X)|\leq |G:H|$,  we have $|P_G(X)|\leq
|G:H||P_H(X)|$ and so
\[|G| \ |\mathcal{L}(G)|  \ pd(G)={\underset{X \in \mathcal{L} (G)}\sum}|P_G(X)|
\leq |G:H|{\underset{X \in \mathcal{L}(G)}\sum}|P_H(X)|\] but,
from \eqref{e:1}, this means
\[=|G:H| \ \left({\underset{X \in
\mathcal{L}(P(G))}\sum}|P_H(X)| + {\underset{X \in \mathcal{L}(G)-
\mathcal{L}(P(G))}\sum}|P_H(X)| \right)\] and the inequality $|\mathcal{L}(G)-\mathcal{L}(P(G))| \le
|\mathcal{L}(P(G))|$ provided by hypothesis, implies
\[\le|G:H| \ \left({\underset{X \in
\mathcal{L}(P(G))}\sum}|P_H(X)| + {\underset{X \in
\mathcal{L}(P(G))}\sum}|P_H(X)| \right) = 2 |G:H| {\underset{X \in
\mathcal{L}(P(G))}\sum}|P_H(X)|\]
\[\le 2|G:H|{\underset{X\in
\mathcal{L}(H)}\sum}|P_H(X)|= 2 \ |G:H| \ |H| \ |\mathcal{L}(H)| \ pd(H)\]
from which we have
$|\mathcal{L}(G)| \ pd(G) \le 2 |\mathcal{L}(H)| \ pd(H).$
\end{proof}

A classic splitting result for the product probability of two
independent events is described by the following corollary. The
proof  may be  generalized to finitely many factors, whose
orders are pairwise coprime.

\begin{prop}\label{c:1} Let $G$ and $H$ be two groups such that $\gcd(|G|,|H|)=1$. Then
$pd(G \times H)= pd(G) \ pd(H)$.
\end{prop}

\begin{proof} Given three groups $A$, $B$ and $C$ such that
$A \times B \subseteq C$,  we
know that $N_C(A \times B)=N_C(A) \times N_C(B)$. This holds
similarly for the permutizers and it is easy to see that
$P_C(A \times B)=P_C(A) \times P_C(B).$  Now this fact and the assumption 
$\gcd(|G|,|H|)=1$ allow us to conclude that
$$ \frac{1}{|G \times H| \ \ |\mathcal{L}(G \times H)|}{\underset{X \times Y \in \mathcal{L}
(G\times H)}\sum}|P_{G\times H}(X \times Y)|$$
$$= \frac{1}{|G| \ |\mathcal{L}(G)|} \ \frac{1}{|H| \
|\mathcal{L}(H)|} \ {\underset{X \in \mathcal{L} (G)}\sum}|P_G(X)|
 \ {\underset{Y \in \mathcal{L} (H)}\sum}|P_H(Y)|.
 $$ \end{proof}

The underlying problem we deal with is the order of the subgroup
lattices, which is hard to predict in general. If
we concentrate on some groups arising from finite geometries, then
the situation is more clear (dihedral groups, semidihedral groups
and generalized quaternion groups were studied in \cite{alghamdi,
elr, far, otera,  mt, mtbis, mttris} from a similar perspective).
Having in mind Examples \ref{d6} and \ref{d8}, we observe from
\cite[pp. 26--29]{rs} and \cite{otera, mt} that the dihedral group
\begin{equation}\label{dihedral}D_{2n}=\langle x,y \ | \ x^2=y^n=1, x^{-1}yx=y^{-1}\rangle= C_2 \ltimes C_n =\langle x \rangle \ltimes \langle y\rangle
\end{equation}
of symmetries of a regular polygon with $n \ge 1$ edges has order
$2n$ and splits in the semidirect product of a cyclic group
$\langle y \rangle \simeq C_n$ of order $n$ by a cyclic group
$\langle x\rangle \simeq C_2$ of order 2 acting by inversion on
$C_n$. In particular,   $S_3 \simeq D_6$ for $n=3$ and one can
note that the Hasse diagram of $\mathcal{L}(D_6)=\mathcal{L}(S_3)$
forms a \textit{diamond} in which there are only 4 \textit{atomic
elements} (see \cite{rs} for this terminology) in between $\{1\}$
and $D_6$ and their number can be easily computed. The following
Fig.III.3 summarizes  the information  of Example \ref{d6} in a
more general situation.

\setlength{\unitlength}{10mm}
\begin{picture}(10,5)(-2.7,-2.5)
\put(1.6,2.1){$D_{2p}$}
\put(2,2){\line(-1,-1){}}
\put(.5,.8){$C_p$}
\put(2,2){\line(0,-1){}}
\put(1,.8){\line(1,-1){}}
\put(1.6,-.5){$\{1\}$}
\put(2,.8){$H_1$}
\put(2,2){\line(1,-1){}}
\put(2,-.2){\line(0,1){}}
\put(2,-.2){\line(1,1){}}
\put(2,-.2){\line(2,1){2}}
\put(3,.8){$H_2$}
\put(6,.8){$H_p$}
\put(2,2){\line(2,-1){2}}
\put(2,-.2){\line(4,1){4}}
\put(2,2){\line(4,-1){4}}
\put(4.2,.8){...}
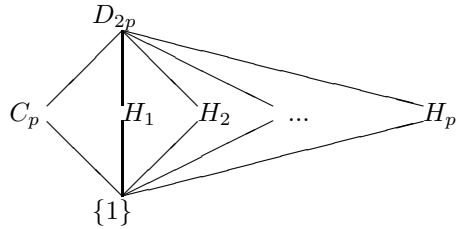
\put(-1,-1){\textbf{Fig.III.3.} Hasse diagram of $\mathcal{L}(D_{2p})$ with odd prime $p \ge 3$;}
\put(0.5,-1.5){$H_1 \simeq H_2 \simeq \ldots \simeq H_p \simeq C_2$.}
\end{picture}

From Figs.III.1, III.2 and III.3, it is clear that
$\mathcal{L}(D_{2p})$ has $p+1$ proper subgroups, and this fact
comes out from the following formula
\begin{equation}\label{lattice}|\mathcal{L}(D_{2n})|=\sigma(n)+\tau(n),\end{equation}
 where $\sigma(n)$ and $\tau(n)$ are the sum and the number of all divisors of $n$ (here $n$ is arbitrary, not necessarily an odd prime), respectively.

In particular, if $n=p^m$ is a power of a prime $p$ (possibly
$p=2$) for some $m\ge0$, then the set  of all divisors of $p^m$ is
$\mathrm{Div}(p^m)=\{1,p,p^2,\ldots,p^m\}$ so that
\begin{equation}
\label{st}\sigma(p^m)=\sum^m_{j=0}p^j=\frac{1-p^{m+1}}{1-p} \ \
\mathrm{and} \ \ \tau(p^m)=|\mathrm{Div}(p^m)|=m+1.
\end{equation}

The reader has probably noted that we used the formula for the sum
of a geometric series in the previous expression for
$\sigma(p^m)$. Then we may conclude that
  \begin{equation}\label{e:2}|\mathcal{L}(D_{2p^m})|=1+m+\frac{1-p^{m+1}}{1-p}=m+\frac{p^{m+1}+p-2}{p-1}.
\end{equation}

The next result shows an upper bound for $pd(G)$, when $|\mathcal{L}(G)|$ is of type \eqref{e:2}.

\begin{thm}\label{l:2} Let $G$ be a noncyclic group and $p$  the smallest prime divisor of $|G|$. If  $|P(G)|=p$ and $|\mathcal{L}(G)|=m+\frac{p^{m+1}+p-2}{p-1}$ for some $m\ge0$, then
$$pd(G)\leq \frac{p^{m+1}+2p^2+(m-3)p-m}{p^{m+2}+(m+1)p^2-(m+2)p}.$$
\end{thm}

\begin{proof}If $P_G(X)=G$ for some $X \in \mathcal{L}(G)$, then
$|G|=|P_G(X)|=|P(G)|$ and $G$ would be cyclic, contradicting our
assumption. Without loss of generality we may assume $P_G(X) \neq
G$ for all $X \in \mathcal{L}(G)$. The minimality of $p$ implies
that $|P_G(X)|\leq \frac{|G|}{p}$ for all $X\in \mathcal{L}(G)$.
Of course, $|\mathcal{L}(P(G))|=2$ and \eqref{e:1} becomes
$$ |G| \cdot \left(\frac{mp-m +p^{m+1}+p-2}{p-1}\right)  \cdot pd(G)
=2|G|+{\underset{X\in \mathcal{L}(G)-\mathcal{L}(P(G))}\sum}|P_G (X)|$$
$$\leq 2|G|+ \left( |\mathcal{L}(G)-\mathcal{L}(P(G))|\right) \frac{|G|}{p}$$$$ =
2|G| + \frac{|G|}{p} (|\mathcal{L}(G)|-2)= 2|G| + \frac{|G|}{p} \left(\frac{mp-m +p^{m+1}+p-2}{p-1}-2\right)$$
$$=\frac{|G|}{p} \left(\frac{mp-m +p^{m+1}+p-2-2p+2+2p^2-2p}{p-1}\right)$$
$$=\frac{|G|}{p} \left(\frac{p^{m+1}+2p^2+(m-3)p-m}{p-1}\right).$$
This gives, as claimed. 
\end{proof}

\section{Some theorems of structure}

The present section is devoted to prove restrictions on $P(G)$, arising from exact bounds for $pd(G)$, when $G$ is an arbitrary group. The evidences of Examples \ref{d6} and \ref{d8} motivated most of the following results.

\begin{thm}\label{t:2}
Let $G$ be  a group (with $pd(G)\not=1$)
and $p$  the smallest prime divisor of $|G|$. Then $$\left(1-\frac{p}{|G|}\right) \frac{|\mathcal{L}(P(G))|}{|\mathcal{L}(G)|}+\frac{p}{|G|} \le pd(G).$$
Moreover, if $P_G(X)$ is a proper subgroup of $G$ for all $X \in
\mathcal{L}(G)-\mathcal{L}(P(G))$, then 
$$ pd(G) \leq \frac{1}{p}+\frac{(p-1)|\mathcal{L}(P(G))|}{p \ |\mathcal{L}(G)|}.$$
\end{thm}

\begin{proof}
In order to prove the lower bound, it is enough to note 
from \eqref{e:1} that
\begin{equation}\label{e:3}|\mathcal{L}(G)||G| pd(G)= |\mathcal{L}(P(G))||G| + {\underset{X \in \mathcal{L}(G)
- \mathcal{L} (P(G))}\sum}|P_G(X)|\end{equation}
$$  \geq |\mathcal{L}(P(G))||G| +
\Big(|\mathcal{L}(G)|-|\mathcal{L}(P(G))|\Big)p, $$
where we have used in the last step that  $| P_G(X)| \geq p$. Then we continue
\[= (|G|-p)|\mathcal{L}(P(G))|+p|\mathcal{L}(G)|\] from which we get
\[ pd(G) \geq \frac{(|G|-p)|\mathcal{L}(P(G))|+p|\mathcal{L}(G)|}{|\mathcal{L}(G)||G|}= \frac{(|G|-p)|\mathcal{L}(P(G))|}{|\mathcal{L}(G)||G|}+\frac{p}{|G|}.\]

Now we prove the upper bound.  Formula \eqref{e:1} becomes again \eqref{e:3}
\[ |G| |\mathcal{L}(G)| \ pd(G) = |\mathcal{L}(P(G))||G| + {\underset{X \in \mathcal{L}(G)
- \mathcal{L} (P(G))}\sum}|P_G(X)|\]
once one uses the fact that $P_G(X)=G$ for every $X \in
\mathcal{L} (P(G))$. Now, since $|G:P_G(X)|\not=1$ for every $X
\in \mathcal{L}(G) - \mathcal{L} (P(G))$, we get $|G:P_G(X)|\geq
p$, that is, $|P_G(X)|\leq \frac{|G|}{p}$. Therefore \eqref{e:3}
is upper bounded by
\[ \leq |\mathcal{L}(P(G))||G| +
\Big(|\mathcal{L}(G)| - |\mathcal{L}(P(G))|\Big) \frac{|G|}{p} =
\frac{(p-1)|\mathcal{L}(P(G))||G| }{p}+
\frac{|\mathcal{L}(G)||G|}{p}\] and the result follows.
\end{proof}

Of course, $D_8$ satisfies the lower bound, but not the upper
bound, of Theorem \ref{t:2}. Details can be deduced from the
information of Example \ref{d8}. This is to justify that Theorem
\ref{t:2} originates from evidences of computational nature. On
the other hand, Example \ref{d8}  shows also that $Z(D_8)=P(D_8)
\simeq C_2$. Then, when can we say that $P(D_8)$ is noncyclic?
The next two results concern this question.

\begin{thm}\label{t:1}
If $P(G)$ is  a nontrivial proper subgroup of a group $G$ and $pd(G)=\frac{1}{2}+\frac{|\mathcal{L}(P(G))|}{2
\ |\mathcal{L}(G)|}$, then $P(G)$ is noncyclic.
\end{thm}

\begin{proof}
By assumption we exclude the cases $P(G)=G$ and $P(G)=\{1\}$,
which are the extremal situations already known. Assume that
$P(G)$ is cyclic of prime order $q\ge p \ge 2$, where $p$ is the
smallest prime dividing  $|G|$.  We may apply the arguments of the
proof of Theorem \ref{l:2} and, noting that
$|\mathcal{L}(P(G))|=2$, we find that  
$$pd(G)=\frac{1}{2}+\frac{|\mathcal{L}(P(G))|}{2 \ |\mathcal{L}(G)|}=
\frac{1}{2}+\frac{1}{|\mathcal{L}(G)|}=\frac{|\mathcal{L}(G)|+2}{2|\mathcal{L}(G)|} \le \frac{2 \ |G|}{|G| \
|\mathcal{L}(G)|} + \frac{|G| \ (|\mathcal{L}(G)|-2)}{p \ |G| \ |\mathcal{L}(G)|}$$
$$=\frac{(2p  +
|\mathcal{L}(G)|-2) \ |G|}{p \ |G| \ |\mathcal{L}(G)|}$$
and then the inequality
$$\frac{|\mathcal{L}(G)|+2}{2}=\frac{|\mathcal{L}(G)|}{2}+1\le \frac{(2p-2)}{p}+ \frac{|\mathcal{L}(G)|}{p}$$
which means
$$\frac{|\mathcal{L}(G)|}{2}-\frac{|\mathcal{L}(G)}{p}=\frac{(p-2) |\mathcal{L}(G)|}{2p}\le 1-
\frac{2}{p}=\frac{p-2}{p}.$$
From this we derive the contradiction  $\frac{|\mathcal{L}(G)|}{2}\le 1$, 
as at least  $\{1\}$ and $G$ are contained in $\mathcal{L}(G)$. 
Therefore, $P(G)$ cannot be
cyclic of prime order and we may assume that $P(G)$ is cyclic of
order $k\ge2$. Now, we note that
$|\mathcal{L}(P(G))|=|\mathrm{Div}(k)|$, where $\mathrm{Div}(k)$
is the set of all divisors of $k$. Here the argument we just used
for $q$ may still be applied. In fact we have
$$\frac{1}{2}+\frac{|\mathrm{Div}(k)|}{2 \ |\mathcal{L}(G)|}=\frac{|\mathcal{L}(G)|+|\mathrm{Div}(k)|}{2|\mathcal{L}(G)|}\le
\frac{|\mathrm{Div}(k)| \ |G|}{|G| \ |\mathcal{L}(G)|} + \frac{|G| \ (|\mathcal{L}(G)|-|\mathrm{Div}(k)|)}{p \
|G| \ |\mathcal{L}(G)|}$$
$$=\frac{(|\mathrm{Div}(k)| \ p  + |\mathcal{L}(G)|-|\mathrm{Div}(k)|) \ |G|}{p \ |G| \
|\mathcal{L}(G)|}$$
then
$$\frac{|\mathcal{L}(G)|+|\mathrm{Div}(k)|}{2}=\frac{|\mathcal{L}(G)|}{2}+\frac{|\mathrm{Div}(k)|}{2}\le
\frac{(p-1)|\mathrm{Div}(k)|}{p}+ \frac{|\mathcal{L}(G)|}{p}$$
which means
$$
\frac{|\mathcal{L}(G)|}{2}-\frac{|\mathcal{L}(G)}{p}=\frac{(p-2) |\mathcal{L}(G)|}{2p}\le \frac{(p-1)
|\mathrm{Div}(k)|}{p}$$
and this would imply that $|\mathcal{L}(G)|\le |\mathrm{Div}(k)|=|\mathcal{L}(P(G))|$, that is, $\mathcal{L}(G)\subseteq
\mathcal{L}(P(G))$ and then $\mathcal{L}(G)=\mathcal{L}(P(G))$. This condition implies $G=P(G)$, a
contradiction.
\end{proof}

The reader may note that Theorem \ref{t:1} describes a very
general situation, which cannot be reduced to those in Examples
\ref{d6} and \ref{d8}. In fact, looking at Example \ref{d6},
$P(D_6)=D_6$ and so $P(D_6)$ is not a proper subgroup of $D_6$,
and this means that one of the assumptions of Theorem \ref{t:1} is
not satisfied. On the other hand,  $P(D_8)=Z(D_8)$ is a
nontrivial proper subgroup of $D_8$ (see Example \ref{d8}), but
$\frac{4}{5}=pd(D_8) \not= \frac{1}{2}+ \frac{1}{10}=\frac{3}{5}$,
and in fact $P(D_8)$ is cyclic. Once again, Theorem \ref{t:1} can
not be applied. These two examples show that we cannot strengthen
further Theorem \ref{t:1}.

However, we may detect groups $G$ with cyclic $P(G)$. The
following result shows this circumstance.

\begin{thm}\label{t:extra}
Let $P(G)$  be a  nontrivial
 proper subgroup of a group $G$ with $pd(G)=\frac{4}{5}$ and $p$  be
 the smallest prime divisor of $|G|$. If 
 $\frac{4|G|-5p}{5|G|-5p}\le \frac{2}{|\mathcal{L}(G)|},$
 then $P(G)$ is cyclic of prime order.
\end{thm}
\begin{proof}From the lower bound of Theorem \ref{t:2}, we have
$$pd(G)=\frac{4}{5} \ge \left(1-\frac{p}{|G|}\right) \frac{|\mathcal{L}(P(G))|}{|\mathcal{L}(G)|}+\frac{p}{|G|} \Leftrightarrow \frac{\frac{4}{5}}{1-\frac{p}{|G|}} - \frac{\frac{p}{|G|}}{1-\frac{p}{|G|}}\ge \frac{|\mathcal{L}(P(G))|}{|\mathcal{L}(G)|}$$
$$\Leftrightarrow \frac{\frac{4|G|-5p}{5|G|}}{\frac{|G|-p}{|G|}} \ge \frac{|\mathcal{L}(P(G))|}{|\mathcal{L}(G)|} \Leftrightarrow \frac{4|G|-5p}{5|G|-5p}\ge \frac{|\mathcal{L}(P(G))|}{|\mathcal{L}(G)|}.$$
We conclude that $\frac{|\mathcal{L}(P(G))|}{|\mathcal{L}(G)|} \le \frac{2}{|\mathcal{L}(G)|},$ hence $|\mathcal{L}(P(G))| \le 2$. This forces $P(G)$ to be cyclic of prime order.
 \end{proof}

\section{Computations for dihedral groups}

We describe an instructive example, which correlates most of the notions
which we have seen until now.

\begin{prop}\label{new1} Let $p$ be an odd prime. Then
$$1=pd(D_{2p}) > sd(D_{2p})=\frac{7p^3-5p^2-11p+9}{p^4+4p^3-2p^2-12p+9}>\frac{p+3}{4p}=d(D_{2p}).$$
\end{prop}

\begin{proof}
Noting  that $D_{2p}=C_2 \ltimes C_p$ (see \eqref{dihedral}) and that $\mathcal{L}(D_{2p})$ forms a diamond
(as in Fig.III.3), we conclude that $Z(D_{2p})$ is trivial and
$C_{D_{2p}}(C_p)=C_p$ is the unique maximal normal subgroup of
$D_{2p}$. Moreover $D_{2p}$ is a $\mathbf{P}$--group, because
$Q_\infty(D_{2p})=D_{2p}$. Therefore a proper subgroup $H$ of
$D_{2p}$ should be properly contained in $P_{D_{2p}}(H)$ and
necessarily $P_{D_{2p}}(H)=D_{2p}$. Thus, we find that
$pd(D_{2p})=1$.

On the other hand, we  may specialize the
formula
$$sd(D_{2p})=\frac{\tau(p)^2+2\tau(p)\sigma(p)+g(p)}{(\tau(p)+\sigma(p))^2},$$
given in \cite[Theorem 3.1.1]{mt}, where
$$g(p)=\frac{3p^3-5p^2+p+1}{p^2-2p+1}$$ is the arithmetic function
in \cite[Eq. 10, p.2514]{mt}. From \eqref{st} we deduce
$\tau(p)=2$ and $\sigma(p)=p+1$ so that
$$sd(D_{2p})=\frac{4+2 \cdot 2 \cdot (p+1)+\frac{3p^3-5p^2+p+1}{p^2-2p+1}}{p^2+6p+9}
=\frac{8+4p+\frac{3p^3-5p^2+p+1}{p^2-2p+1}}{p^2+6p+9}$$
$$=\frac{\frac{8p^2-16p+8+4p^3-8p^2+4p+3p^3-5p^2+p+1}{p^2-2p+1}}{p^2+6p+9}$$$$=\frac{7p^3-5p^2-11p+9}{(p^2+6p+9)(p^2-2p+1)}=\frac{7p^3-5p^2-11p+9}{p^4+4p^3-2p^2-12p+9}.$$
Now \cite[Remark 4.2]{les2} shows that $d(D_{2p})=\frac{p+3}{4p}$ and for all odd primes  we have
\[0>p^5-21p^4+30p^3+26p^2-63p+27\]
\[ \Leftrightarrow 28p^4-20p^3-44p^2+36p>p^5+7p^4+10p^3-18p^2-27p+27\]
\[ \Leftrightarrow 28p^4-20p^3-44p^2+36p>p^5+4p^4-2p^3-12p^2+9p+3p^4+12p^3-6p^2-36p+27\]
\[\Leftrightarrow 4p(7p^3-5p^2-11p+9)>(p+3)(p^4+4p^3-2p^2-12p+9)\]
\[\Leftrightarrow \frac{7p^3-5p^2-11p+9}{p^4+4p^3-2p^2-12p+9}>\frac{p+3}{4p}.\]
The result follows.
\end{proof}

From  Proposition \ref{new1}, $pd(D_{2p})$ has
a constant value for all odd primes, while $sd(D_{2p})$ and
$d(D_{2p})$ are functions of $p$. This is an important difference
of the permutability degree with respect to the subgroup
commutativity degree and the commutativity degree. This reflects the fact that we are looking at
permutizers in a group, and not at centralizers.

\section{Some open questions}

We end with a series of open questions. They arise naturally from the study of the present subject. The first is motivated by the families of dihedral groups, analyzed in Section  6 (and in other parts of the present paper). Most of these groups may be described  in terms of product of groups.

\begin{oq} Let $G=NH$ be a product of a normal subgroup $N$ by a subgroup $H$. What can be said about the permutability degree of $G$ ?
\end{oq}

It is well known  that most of dihedral groups, generalized quaternion groups and semidihedral groups has this structure. They present further analogies in terms of central quotients: one, for instance, is that $D_8/Z(D_8) \simeq Q_8$ and, roughly speaking, one can generalize this isomorphism to $Q_{2^4}=Q_{16}$, $Q_{2^5}=Q_{32}$ and so on. The reader may refer to  \cite{hr1, hr2} for  recent studies on these groups. Therefore:

\begin{oq} What is  the permutability degree of  generalized
quaternion groups and semidihedral groups ?
\end{oq}

There is also another question, which is more general and may require some computational efforts. A classical result of Cayley allows us to embedd a  group in a suitable symmetric group. The knowledge of symmetric groups plays in fact a fundamental role in several aspects of the theory of groups. Therefore:

\begin{oq} What is the permutability degree of the symmetric group $S_n$ on $n$ objects?
\end{oq}

Finally, the classification of (finite) simple groups may provide interesting aspects of study. 
For the (finite) simple (and almost simple) groups of sporadic type a lot is
known about their subgroup lattices, see \cite{connor}. Therefore

\begin{oq}What is the permutability degree of  (finite) simple groups ?
\end{oq}

\section*{Acknowledgements}
The first author was partially supported by GNSAGA of INdAM (Italy). 
The research of the second author is supported in part by NRF  (South Africa) for the Grant No. 93652 and in part from the Launching Grant  No. 459235 of the University of Cape Town (South Africa). He also thanks Universidade Federal  do Rio de Janeiro (Rio de Janeiro, Brazil), where  the present subject was originally investigated via the project CAPES No. 061/2013. Finally, we thank the editor for his valuable work and the referee for the questions in Section 7.

\newpage

\section{ERRATA CORRIGE  -- 10TH OF SEPTEMBER 2017}

After the publication of the previous results in:\\
\\
 D.E. Otera and F.G. Russo, Permutability degrees of finite groups, Filomat 30 (2016), 2165--2175. \\
\\
some errors have been reported to our attention, thanks to the communication of some colleagues  (in particular we thank Prof. Paz Jim\'enez Seral). As noted below, the main results are not directly involved, but  we cannot avoid to give more details  on some points which might be ambiguous if we don't do it.

\begin{itemize}
\item[(1)] Definition of $P(G)$ in Introduction. Two lines below. Replace the sentence ``Note that for any $X \in \mathcal{L}(P(G))$ one has $P_G(X)=G$'' by ``Note that when $X=P(G)$ one has $P_G(X)=P(G)$''.
\item[(2)] The value of $pd(D_8)$ in Example 3.2 is wrong. In fact $pd(D_8)=1$, in agreement with  Proposition 6.1 when we show that $pd(D_{2p})=1$ for dihedral groups (of order $2p$ with $p$ odd). The error of Example 3.2 is here:
$$4=|M_1|=|P_{D_8}(H)|=|P_{D_8}(K)|=|P_{D_8}(V)|=|P_{D_8}(U)|,$$
 in fact we must replace this computation with
$$8=|P_{D_8}(H)|=|P_{D_8}(K)|=|P_{D_8}(V)|=|P_{D_8}(U)|.$$
Consequently, $P(D_8)=D_8$ and  $P(D_8) \neq Q(D_8)$. In particular, the final two sentences from ``This example shows ...'' until the end of Example 3.2 must be  removed. In addition,
\begin{itemize}
\item[(a)] the final sentence of Remark 2.3 must be removed;
\item[(b)] the sentence just before Theorem 5.1 must be removed;
\item[(c)] the sentences before Theorem 5.2 ``Of course, $D_8$ ... $P(D_8) \simeq C_2$ '' must be removed.
\item[(d)] the sentences ``On the other hand... further Theorem 5.2'' before Theorem 5.3 must be removed.
\end{itemize}

\item[(3)] Proof of the lower bound in Theorem 4.1. There is a misprint: it must be 
$$\sum_{X \in \mathcal{L}(G)}|P_H(X)| \ge \sum_{X \in \mathcal{L}(H)}|P_H(X)|$$

\item[(4)]In Theorem 4.3 we use an assumption, not justified properly in the proof, but this may be motivated by evidences of computational nature. Therefore Theorem 4.3 must be reformulated in the following way:

\begin{thm} Let $G$ be a noncyclic group and $p$  the smallest prime divisor of $|G|$. If  $P_G(X)$ is a proper subgroup of $G$ for all $X \in \mathcal{L}(G)$,  $|P(G)|=p$ and $|\mathcal{L}(G)|=m+\frac{p^{m+1}+p-2}{p-1}$ for some $m\ge0$, then
$$pd(G)\leq \frac{p^{m+1}+2p^2+(m-3)p-m}{p^{m+2}+(m+1)p^2-(m+2)p}.$$
\end{thm}

The proof  is  that of Theorem 4.3 without the first two sentences.


\end{itemize}

\end{document}